\documentclass[letterpaper, 10 pt, conference]{ieeeconf}  
\IEEEoverridecommandlockouts                              
\usepackage{graphics} 
\usepackage{graphicx}  
\usepackage{amsmath}  
\usepackage{amssymb}  
\usepackage{amsthm} 
\usepackage{algpseudocode}
\usepackage{algorithm}
\usepackage{algorithmicx} 
\usepackage{setspace}
\usepackage{color}
\usepackage{verbatim} 
\usepackage[font= footnotesize]{caption}

\theoremstyle{plain} 
\newtheorem{assumption}{Assumption}
\theoremstyle{plain} 
\newtheorem{lem}{Lemma}
\theoremstyle{theorem} 
\newtheorem{thm}{Theorem}
\theoremstyle{remark}
\newtheorem{rem}{Remark}
\DeclareMathOperator*{\argmin}{argmin}
\DeclareMathOperator*{\trans}{T}
\newcommand{\ahmpc}{AH-MPC}
\title{\LARGE \bf Constrained LQR Using Online Decomposition Techniques{$^*$}}

\author{L. Ferranti$^{1}$, G. Stathopoulos$^{2}$, C. N. Jones$^2$, and T. Keviczky$^1$
\thanks{*This research is supported by the European Union's Seventh Framework
Programme (FP7/2007-2013) under grant agreement n. AAT-2012-RTD-2314544
(RECONFIGURE), by the TU Delft Space Institute, by the European Union's Seventh
Framework Programme (FP/2007-2013)/ ERC grant agreement n. 307608 (BuildNet), and by the People Programme (Marie Curie Actions) of the European Union's Seventh Framework Programme (FP7/2007-2013) under REA grant agreement n. 607957 (TEMPO).}
 \thanks{{$^{1}$}L. Ferranti and T.
Keviczky are with the Delft Center for Systems and Control, Delft University of Technology, Delft, 2628 CD, The Netherlands,
        {\tt\small $\{$l.ferranti,t.keviczky$\}$@tudelft.nl}}%
\thanks{{$^{2}$}G. Stathopoulos and C. N. Jones are with the Laboratoire d'Automatique, \'{E}cole Polytecnique F\'{e}d\'{e}rale de Lausanne (EPFL), Lausanne, CH-1015, Switzerland,
        {\tt\small $\{$georgios.stathopoulos,colin.jones$\}$@epfl.ch}}}%

\newif\ifpaper 
\papertrue 
\begin{document}

\maketitle
\thispagestyle{empty}
\pagestyle{empty}

\begin{abstract}
This paper presents an algorithm to solve the infinite horizon constrained linear quadratic regulator (CLQR) problem using operator splitting methods. 
First, the CLQR problem is reformulated as a (finite-time) model predictive control (MPC) problem without terminal constraints. Second, the MPC problem is decomposed into 
smaller subproblems of fixed dimension independent of the horizon length. Third, using the fast alternating minimization algorithm to solve the subproblems, the horizon length 
is estimated online, by adding or removing subproblems based on a periodic check on the state of the last subproblem to determine whether it belongs to a given control invariant set. 
We show that the estimated horizon length is bounded and that the control sequence computed using the proposed algorithm is an optimal solution of the CLQR problem. 
Compared to state-of-the-art algorithms proposed to solve the CLQR problem, our design solves at each iteration only unconstrained least-squares problems and simple gradient
 calculations. Furthermore, our technique allows the horizon length to decrease online~(a useful feature if the initial guess on the horizon is too conservative). 
 Numerical results on a planar system show the potential of our algorithm.
\end{abstract}
\section{INTRODUCTION}
\label{sec:sec_introduction}
The linear quadratic regulator (LQR) proposed by~\cite{Kalman1960} allows one, under mild assumptions on the system dynamics, to design an optimal state feedback to stabilize the
 plant in closed loop in the absence of constraints. When constrains are present the controller should be able to exploit as much as possible the actuator operating ranges to maximize the production. This observation motivated the study of model predictive control (MPC) in the late 70s~\cite{Richalet1978,Cutler1980}. 

An MPC controller solves a constrained optimization problem (derived from the LQR formulation) over a finite time window (prediction horizon). The main advantage of MPC is its
 ability to handle constraints. The use of a finite time window, however, compromises the nice properties of the LQR controller in terms of optimality and closed-loop stability
  (more details can be found in~\cite{Maciejowski2002,Borrelli2015}). Closed-loop stability of the MPC controller can be preserved by including in the MPC problem 
  formulation a terminal set (together with a terminal cost), as discussed
  in~\cite{Mayne2000}. The main advantage is that this formulation is equivalent
  to solving the constrained infinite horizon LQR problem gaining tractability
  from a computational point of view. In particular, the terminal set is used to
  constrain the last predicted state to remain within a control invariant set.
   This set is usually selected to be the maximal positively
   invariant set of the closed loop obtained using the (unconstrained) LQR
   control law associated with the MPC controller. On one hand, the use of the terminal set
   (together with a terminal cost) allows one to prove closed-loop stability of
   the system controlled with MPC. On the other hand, the use of the terminal set reduces the region of attraction of the MPC
   controller leading, in general, to more conservative performance. For this reason, many practical MPC applications rely on the use of a \emph{sufficiently long} prediction horizon tuned offline to ensure that the system will converge to the 
   terminal set. Although this approach often works in practice, there are no guarantees that for all the possible initial conditions the closed-loop system is stable.

\textit{\textbf{Contribution.}} We aim to solve the CLQR problem in a
computationally tractable way, without relying on the use of the terminal set.
In this respect, we propose the use of decomposition techniques to exploit the MPC problem structure and estimate the length of the prediction horizon online.
 The constrained LQR problem is reformulated as an MPC problem. The length of the horizon in the MPC problem, however, is not fixed, but it is decided online 
 by our algorithm to guarantee closed-loop stability. By relying on the fast alternating minimization algorithm (FAMA)~\cite{Goldstein2014}, our algorithm solves at each iteration
  unconstrained least-squares problems and simple gradient calculations of size independent of the length of the prediction horizon. We show, under mild assumptions on the cost and 
  on the system dynamics, that the length of the horizon is bounded and the solution using our algorithm is an optimal solution of the CLQR problem. Finally, numerical results using 
  the planar system of~\cite{StathopoulosTAC2015} are provided to show the
  potential of our proposed approach.

\textit{\textbf{Related Work.}} The proposed approach relies on the work of~\cite{Scokaert1998} for the CLQR. We combined their technique with the decomposition along the length of the
 prediction horizon (\emph{time-splitting} approach) proposed in~\cite{StathopoulosECC13}, aiming to reduce the computational complexity of the algorithm, as detailed in 
 Section~\ref{sec:sec_clqr_time_splitting}.
In~\cite{StathopoulosTAC2015}, an algorithm to compute online the length of the horizon using operator-splitting techniques is also proposed. Compared to their approach, we use a 
different splitting technique that allows the length of the horizon to decrease online aiming to reduce the conservatism in the initial guess on the horizon length.  

\textit{\textbf{Outline.}} Section~\ref{sec:problem_formulation} introduces our
problem formulation. Section~\ref{sec:sec_preliminaries} summarizes existing
results from~\cite{Goldstein2014} and~\cite{Scokaert1998}. Section~\ref{sec:sec_clqr_time_splitting} details our proposed approach.
Section~\ref{sec:sec_simulation_results.tex} presents numerical results. Section~\ref{sec:sec_conclusions} concludes the paper.
 
\textit{\textbf{Notation.}} For $u\in \mathbb R^n$, $\left \| u \right
\|=\sqrt{\langle u,u \rangle}$ is the Euclidean norm. Let $\mathbb C$ be a convex set. Then, $\mathbf{Pr}_{\mathbb C}(u)$ is the projection of $u$ onto $\mathbb C$. 
Furthermore, $\mathcal I_{\mathbb C}(\sigma)$ is the indicator function on the convex set $\mathbb C$, which is zero if $\sigma \in \mathbb C$ and infinity otherwise.
Let $A\in\mathbb R^{n\times m}$. Then, $\operatorname{eig}_{\max}(A)$ and $\operatorname{eig}_{\min}(A)$ are the largest and the smallest (modulus) eigenvalues of $A^{\trans}A$. $P  \in\mathbb S_{++}^{n\times n}$ denotes that $P=P^{\operatorname{T}}\in \mathbb R^{n\times n}$ is positive definite. Finally, details on the notions of strong convexity and Lipschitz continuity can be found in~\cite{Boyd2004}. 
 
\section{PROBLEM FORMULATION}
\label{sec:problem_formulation}
This section presents the constrained LQR (CLQR) problem that we aim to solve using the decomposition techniques proposed in Section~\ref{sec:sec_clqr_time_splitting}.

Consider the discrete linear time-invariant (LTI) system described as follows:
\begin{equation}
\label{eq:LTIsystem}
x(t+1) = Ax(t)+Bu(t),
\end{equation}
where the state $x(t)\in \mathbb R^n$ and control input $u(t)\in\mathbb R^m$ are subject to the following constraints:
\begin{equation}
\label{eq:constraints}
Cx(t)+Du(t)\leq d,
\end{equation} 
and $A\in \mathbb R^{n\times n}$, $B\in  \mathbb R^{n\times m}$, $C\in\mathbb R^{p\times n}$, and $D\in \mathbb R^{p\times m}$ are known constant matrices.
\begin{assumption}
\label{ass:stabilizability}
The pair $(A,B)$ is stabilizable.
\end{assumption}
Our goal is to regulate to the origin the state of the system, starting from a feasible initial condition. In the absence of constraints, a natural choice to achieve this goal would be to design an optimal state feedback using the well known unconstrained linear quadratic regulator (ULQR) theory~\cite{Kalman1960}. The design of the optimal state feedback, however, is complicated by the presence of the constraints~\eqref{eq:constraints}. In particular, the computation of an optimal control law to regulate the system to the origin leads to the following infinite-dimensional optimization problem:
\begin{subequations}
\label{eq:clqr}
\begin{align}
\label{eq:infinite_cost_function}
\underset{x,u}{\operatorname{minimize}}~& \frac{1}{2}\sum\limits_{t=0}^{\infty}x_t^{\textrm{T}}Qx_t+u_t^{\textrm{T}}Ru_t\\\label{eq:clqr_equality_constraints}
\operatorname{subject~to:~}& x_{t+1} = A x_t+B u_t,\quad t\in\mathbb N\\
& x_{0} = x_{\textrm{init}},\\
& d~\geq Cx_t+D u_t,\quad t\in \mathbb N,
\end{align}
\end{subequations}
where $x_t$ and $u_t$ are the $t$-step-ahead state and control predictions, respectively.    
We refer to Problem~\ref{eq:clqr} as the CLQR problem.
\begin{assumption}
\label{ass:positive_def}
$Q\in\mathbb S^{n\times n}_{++}$ and $R\in\mathbb S^{m\times m}_{++}$. 
\end{assumption}
The CLQR control law obtained by solving Problem~\eqref{eq:clqr} is $u^{\infty}
= u_0$, where $u_0(x_{\textrm{init}})$ is the first element of the infinite
sequence of predicted control commands $\mathbf u^{\infty}:=\{u_0,u_1,\ldots\}$.

Suppose that Assumptions~\ref{ass:stabilizability} and~\ref{ass:positive_def} hold. Then, there exists an optimal stabilizing state-feedback gain $K\in \mathbb R^{m\times n}$ and 
$P\in\mathbb S^{n\times n}_{++}$ (solution of the algebraic Riccati equation) associated with the ULQR. Furthermore, starting from a given initial condition $x_{\textrm{init}}$, there exists a time instance $N^{\infty}$ in 
which the state will enter a polyhedral set $\mathcal X_f := \left\{x\in\mathbb R^n\,|\,C_fx\leq d_f\right\}$ such that $\forall$ $x(t)\in \mathcal X_f~\Rightarrow$~$x(t+1) = (A+BK)x(t) \in \mathcal X_f$ $\forall t\geq N^{\infty}$. The computation of $N^{\infty}$ is, in general, very challenging and in classical MPC it is common practice to reformulate Problem~\eqref{eq:clqr} as follows:
\begin{subequations}
\label{eq:FH-MPC}
\begin{align}
\label{eq:finite_cost_function}
\underset{x,u}{\operatorname{minimize}}~& \frac{1}{2}\sum\limits_{t=0}^{N-1}x_t^{\textrm{T}}Qx_t+u_t^{\textrm{T}}Ru_t+V_f\\
\operatorname{subject~to:~}& x_{t+1} = A x_t+B u_t,\quad t=0,\ldots,N\\
& x_{0} = x_{\textrm{init}},\\
& d~\geq Cx_t+D u_t,\quad t=0,\ldots, N-1,\\
& x_{N}\in X_f,
\end{align}
\end{subequations}
where the horizon $N$ is fixed based on some heuristic, $X_f$ is used as terminal set to constrain the $N$-step-ahead predicted state and {$V_f:=x_{N}^\textrm{T}Px_{N}$} is used as terminal cost to replace {$\sum_{N^{}}^\infty (\cdot)$} in the cost~\eqref{eq:infinite_cost_function}. We refer to Problem~\eqref{eq:FH-MPC} as the finite-horizon MPC (FH-MPC) Problem. 

The use of the terminal set allows one to prove closed-loop stability of the
proposed MPC controller, but it also leads to conservatism in the performance.
To avoid this conservatism, in practical applications, it is common practice to
remove the terminal set and define a sufficiently long horizon $N^{}$ offline to
guarantee that the predicted state at the end of the horizon is inside a positively invariant set, which is hard to compute.

Our work aims to solve Problem~\eqref{eq:clqr} by using an MPC approximation of the CLQR problem without explicitly relying on the terminal set. The proposed MPC controller relies on an online estimation strategy of the horizon length $N\geq N^{\infty}$ to ensure that the predicted terminal state enters the terminal set. In particular, we aim to solve the following~problem:
\begin{subequations}
\label{eq:IH-MPC}
\begin{align}
\label{eq:IH-MPC_cost}
\underset{x,u}{\operatorname{minimize}}~& \frac{1}{2}\sum\limits_{t=0}^{N-1}x_t^{\textrm{T}}Qx_t+u_t^{\textrm{T}}Ru_t+V_f\\
\operatorname{subject~to:~}& x_{t+1} = A x_t+B u_t,\quad t=0,\ldots,N\\
& x_{0} = x_{\textrm{init}},\\
\label{eq:IH-MPC_inequality_constraints}
& d~\geq Cx_t+D u_t,\quad t=0,\ldots, N.
\end{align}
\end{subequations}
In the remainder of the paper, we refer to Problem~\eqref{eq:IH-MPC} as the adaptive-horizon MPC (\ahmpc), which differs from the FH-MPC Problem~\eqref{eq:FH-MPC} in the 
definition of constraints (no terminal constraints) and length of $N\geq N^{\infty}$ adapted online.
 
Our approach is based on a similar idea to the one proposed in~\cite{Scokaert1998} combined with the use of splitting methods~\cite{Goldstein2014}. The next section summarizes these approaches.  
\section{PRELIMINARIES}
\label{sec:sec_preliminaries}
In the following, we review interesting and closely related existing strategies to solve the CLQR Problem~\eqref{eq:clqr}.
\subsection{Constrained Linear Quadratic Regulator~\cite{Scokaert1998}}
\label{subsec:scokaert}
The design described in Section~\ref{sec:sec_clqr_time_splitting} strongly relies on the results proposed by~\cite{Scokaert1998}. Hence, in the following, we summarize their algorithm and their main findings useful for the scope of the paper. 

Under Assumptions~\ref{ass:stabilizability} and~\ref{ass:positive_def} the following holds:
\begin{thm}[Theorem 3 in~\cite{Scokaert1998}]
\label{thm:scokaert}
Let $V^{\infty}$ be the optimal cost obtained by solving the CLQR Problem~\eqref{eq:clqr} and let $\mathbf u^{\infty}$ be the associated optimal control sequence, for $x_0=x_{\textrm{init}}$. 
Furthermore, let $V^{\textrm{\ahmpc}}$ be the optimal cost obtained by solving the \ahmpc~Problem~\eqref{eq:IH-MPC} using~Algorithm~\ref{alg:scokaert} and let
 $\mathbf u^{\textrm{\ahmpc}}$ be the associated optimal control sequence. Then,
 for every $x_0\in\mathcal X$, where $\mathcal X$ indicates the set of feasible
 states for which $V^{\infty}$ is finite, there exists a \textbf{finite}
 positive integer $N^{\infty}(x_0)$ such that $V^{\infty} = V^{\textrm{\ahmpc}}$ and $\mathbf u^{\infty} = \mathbf u^{\textrm{\ahmpc}}$ for all $N\geq N^{\infty}$.
\end{thm} 

By relying on the results of the theorem above, the authors in~\cite{Scokaert1998} propose Algorithm~\ref{alg:scokaert} to solve Problem~\eqref{eq:IH-MPC} (and consequently Problem~\eqref{eq:clqr}). Note that step~1 of Algorithm~\ref{alg:scokaert} requires the computation of a solution for Problem~\eqref{eq:IH-MPC} until the optimal length of the horizon has been computed. This might be very expensive from the computational viewpoint, if the initial estimate of the horizon is too conservative. In~\cite{Scokaert1998}, the authors suggest to use $N^0 = 0$ to reduce the computational load. Section~\ref{sec:sec_clqr_time_splitting} shows how to overcome this issue by exploiting the structure of Problem~\eqref{eq:IH-MPC}. 
\captionsetup[algorithm]{font=normal}
\begin{algorithm}[t]
 \fontsize{8}{8}\selectfont
\begin{algorithmic}  
 \State{Given $N= N^{0}$, $K$, $\mathcal X_f$, $x_{\textrm{init}}$.}
\State{{1.} Solve the \ahmpc~Problem~\eqref{eq:IH-MPC}.}
\If {$x_{N}\notin \mathcal X_f$}
\State{{2.} $N = N+1$.}
\State{3. Return to Step 1.}
\EndIf
\vspace{0.01in}
\State{$\mathbf u^{\textrm{\ahmpc}}=\{u_0,\ldots,u_{N-1},Kx_{N},Kx_{N+1},\ldots\}$.}
 \caption{Constrained LQR~\cite{Scokaert1998}.}
 \label{alg:scokaert}
 \end{algorithmic} 
\end{algorithm}  

\subsection{Fast Alternating Minimization Algorithm}
\label{subsec:fama}
Our design relies on the use of splitting methods. In particular, we exploit the fast alternating minimization algorithm (FAMA)~\cite{Goldstein2014}. FAMA solves the following problem:
\begin{subequations}
\label{eq:fama}
\begin{align}
\operatorname{minimize}~&f(x)+g(y)\\
\operatorname{subject~to:~}& H_xx+H_y y= d.
\end{align}
\end{subequations}
The functions $f$ and $g$ satisfy the following assumptions:
\begin{assumption}
\label{ass:fama_1}
$f$ is strongly convex with convexity parameter $\sigma_f$.
\end{assumption}
\begin{assumption}
\label{ass:fama_2}
$g$ is a convex function not necessarily smooth.
\end{assumption}
FAMA (described in Algorithm~\ref{alg:fama}) is equivalent to apply the fast
proximal gradient method (e.g., FISTA~\cite{Beck2009}) on the dual function of
Problem~\eqref{eq:fama}, as detailed in~\cite{Goldstein2014}.
\captionsetup[algorithm]{font=normal}
\begin{algorithm}[t]
 \fontsize{8}{8}\selectfont
\begin{algorithmic}  
 \State{Given $H_x$, $H_y$, $c$, $\mu_{\textrm{init}}$, $\tau<{\sigma_f}/{\operatorname{eig}_{\max}(H_x)}$.}
 \State{Initialize $\hat \mu^1 = \mu^0 = \mu_{\textrm{init}}$, $\alpha^0=1$, $\alpha^1 = (1+\sqrt{5})/2$.}
\For{$k=1, 2\ldots$}
\State{{1.} $x^{k} = \argmin_x f(x)+\langle\hat\mu^k,-H_xx \rangle$.}
\State{{2.} $y^{k} = \argmin_y g(y)+\langle\hat\mu^k,-H_yy \rangle+\frac{\tau}{2}\|d-H_xx^k-H_yy\|^2$.}
\State{{3.} $\mu^k = \hat \mu^k +\tau(d-H_xx^k-H_yy^k)$.}
\State{{4.} $\alpha^{k+1} = (1+\sqrt{4{\alpha^{k}}^2+1})/2$.}
\State{{5.} $\hat \mu^{k+1} = \mu^k + (\alpha^k-1)(\mu^k-\mu^{k-1})/\alpha^{k+1}$.}
\EndFor
 \caption{Fast alternating minimization algorithm~\cite{Goldstein2014}.}
 \label{alg:fama}
 \end{algorithmic}  
\end{algorithm}  
FAMA can handle problems in which $x\in\mathbb C$, i.e., when $x$ belongs to
a convex set $\mathbb C$. Both in the unconstrained and constrained scenarios,
if Assumptions~\ref{ass:fama_1} and~\ref{ass:fama_2} are satisfied, it is possible to show that the FAMA has a theoretical convergence rate to the optimal solution for Problem~\eqref{eq:fama} of $\mathcal O(\frac{1}{k^2})$~(refer to \cite{Goldstein2014,Pu2014} for more details).
\ifpaper~
The following lemmas are useful for the remainder of the paper:
\begin{lem}
\label{lem:Pu1}[Lemma 3.4 in~\cite{Pu2014}]
Let $\mathbb C$ be a convex cone. The conjugate function of the indicator function of the set $\mathbb S := \{v|-v\in\mathbb C\}$ is equal to the indicator function of the dual cone of $\mathbb C$, i.e., $\mathcal I_{\mathbb S}^{\star} (v) = \mathcal I_{\mathbb C^{\star}}(v)$.
\end{lem}
\begin{lem}
\label{lem:Pu2}[Lemma 3.5 in~\cite{Pu2014}]
Let $\mathbb C$ be the nonnegative orthant $\mathbb C :=\{v\,|\,v\geq 0\}$. For any $v\in \mathbb R^{N_{\mathbb C}}$, the point $z\in\boldsymbol{\operatorname{Pr}}_{\mathbb C}(v)-v$ satisfies $z\in\mathbb C$.
\end{lem}
\else 
\fi
\section{CONSTRAINED LQR WITH ADAPTIVE DECOMPOSITION ALONG THE HORIZON}
\label{sec:sec_clqr_time_splitting}
This section presents our proposed approach to solve Problem~\eqref{eq:clqr}.
In particular, compared to the decomposition technique used
in~\cite{StathopoulosTAC2015}, our approach relies on the decomposition along
the length of the prediction horizon proposed in~\cite{StathopoulosECC13} that
allows one to solve $N+1$ smaller subproblems (in place of
Problem~\eqref{eq:clqr}) of size independent of the length of the horizon, which is an appealing quality when this quantity is unknown and potentially large.

According to the time-splitting strategy presented in~\cite{StathopoulosECC13}, we introduce a new set of decision variables $z_t$ ($t = 1,\ldots,N$) 
to break up the dynamic coupling~\eqref{eq:clqr_equality_constraints}. Furthermore, we add a new set of decision variables $\sigma_t$ ($t=0,\ldots,N$) to handle the inequality 
constraints~\eqref{eq:IH-MPC_inequality_constraints}. Then, we reformulate Problem~\eqref{eq:IH-MPC} as follows:
\begin{subequations}
\label{eq:clqr_time_splitting}
\begin{align}
&\min\limits_{x,u}~\frac{1}{2}\sum\limits_{t = 0}^{N}
x_t^{(t)^{\textrm{T}}}Q_tx_t^{(t)}+u_t^{(t)^{\textrm{T}}}R u_t^{(t)}+\mathcal
I_{\mathbb C}(\sigma_t)\\
\label{eq:time_splitting_MPC_problem_consensus_1}
&\textrm{s.t.:~} z_{t+1} = Ax_t^{(t)}+Bu_t^{(t)} \quad t = 0,\ldots, N-1\\
\label{eq:time_splitting_MPC_problem_consensus_2}
&~\quad~ z_{t+1} = x_{t+1}^{(t+1)} \quad\quad\quad\quad~ t = 0,\ldots,N-1\\
\label{eq:time_splitting_MPC_problem_ineq_constraints}
&~\quad~ \sigma_t =d - Cx_t^{(t)}-Du_t^{(t)} \quad\quad~ t=0,\ldots,N\\
&~\quad~ x_0^{(0)} = x_{\textrm{init}},
\end{align}
\end{subequations}
where, defining $\mathbb C = \{\sigma\in \mathbb R^{p}\,|\,\sigma\geq 0\}$, the indicator functions $\mathcal I_{\mathbb C}(\sigma_t)$ has been added to the 
cost~\eqref{eq:IH-MPC_cost} to 
penalize the inequality constraint
violations~\eqref{eq:IH-MPC_inequality_constraints}, and $Q_t = Q$ for $t
=0,\ldots,N-1$ and $Q_t = P$ for $t=N$. Note that if the consensus constraints
are satisfied, Problem~\eqref{eq:clqr_time_splitting} is equivalent to Problem~\eqref{eq:IH-MPC}. Hence, this implies that computing $N$ for Problem~\eqref{eq:clqr_time_splitting}
 is equivalent to computing $N$ for Problem~\eqref{eq:IH-MPC}. Consequently, the results of Theorem~\ref{thm:scokaert} hold and we can rely on the existence of a finite value 
 $N^\infty\leq N$, such that $x_{N^{\infty}}$ is in the terminal set $X_f$.
 
In the following, first, we show that we can use FAMA to solve Problem~\eqref{eq:clqr_time_splitting}. Second, we present our proposed algorithm to compute $N$ and a solution for
 Problem~\eqref{eq:IH-MPC}. Finally, we show that the control sequence obtained using the proposed algorithm is an optimal solution of Problem~\eqref{eq:clqr}. 

Let $f(\mathbf{x,u}) = \sum_{t=0}^{N^{}} 
x_t^{(t)^{\textrm{T}}}Q_tx_t^{(t)}+u_t^{(t)^{\textrm{T}}}R u_t^{(t)}$ and
$g(\boldsymbol{\sigma,z}) = \sum_{t = 0}^{N^{}} \mathcal I_{\mathbb C} (\sigma_t)$.
If Assumption~\ref{ass:positive_def} is satisfied, $f(\mathbf{x,u})$ satisfies Assumption~\ref{ass:fama_1} with $\sigma_f = \operatorname{eig}_{\min} 
(\operatorname{blockdiag}\{Q,R\})$. In addition, given that $\mathbb C$ is a convex set and the associated indicator function is convex, $g(\boldsymbol{\sigma,z})$ satisfies 
Assumption~\ref{ass:fama_2}. Hence, we can rely on FAMA to solve Problem~\eqref{eq:clqr_time_splitting}. In particular, given that FAMA operates on the dual space, 
we formulate the dual of Problem~\eqref{eq:clqr_time_splitting} as follows:
\begin{equation}
\underset{{\boldsymbol \mu}}{\operatorname{maximize}}~ D(\boldsymbol \mu),
\end{equation}
where the dual function is defined as follows:
\begin{align*} 
\mathbf{D}(\boldsymbol \mu) =& \min_{\mathbf y,\boldsymbol \sigma, \mathbf z} f(\boldsymbol y)+g(\boldsymbol {\sigma,z})+\sum_{t=0}^{N^{}}\langle\lambda_t,-Gy_t-\sigma_t+d\rangle\\
&+\sum_{t=1}^{N^{}} \langle w_t, z_t-H_1y_t\rangle+\langle v_{t},z_t-H_2y_{t-1}\rangle,
\end{align*}
 $\boldsymbol \mu^{\trans} = [w_1^{\trans},\ldots,w_{N}^{\trans},v_1^{\trans},\ldots,v_{N}^{\trans},\lambda_0^{\trans},\ldots,\lambda_{N}^{\trans}]$, $\mathbf y^{\trans} = [y_0^{\trans},\ldots,y_{N}^{\trans}]$, $y_t^{\trans}=[x_t^{(t)^{\trans}},u_t^{(t)^{\trans}}]$, $\boldsymbol\sigma^{\trans} =[\sigma_0^{\trans},\ldots,\sigma_{N^{}}^{\trans}]$, $\mathbf z^{\trans} =[z_1^{\trans},\ldots,z_{N^{}}^{\trans}]$, $H_1 = [I_n~0]$, $H_2 = [A~B]$, $G=[C~D]$. 

\captionsetup[algorithm]{font=small}
\begin{algorithm}[t]
\fontsize{8}{8}\selectfont
\begin{algorithmic} 
\State{Given $x_{\textrm{init}}$, $N_{\textrm{init}}$, $H_1$, $H_2$, $\mathcal Q$, $G$, $g$, $\tau$, $\eta$, and $\bar k$, $\mathcal X_f$, $K$, and $\tilde H$.}
\State{Set $N^0 = N_{\textrm{init}}$, $\alpha^0=1$, $\alpha^1 =
{(1+\sqrt{5})}/{2}$, $\mu_t^0=\hat \mu_t^1=\mu_t^{\textrm{start}}$.}
\For{$k= 0,\ldots, \bar k$}
  \State{{1.a}~For $t=0$, $u_0^{{k+1}} =\argmin_{u_0}~\mathcal L_0$ and  $y_0^{k+1}:=\left[x_{\textrm{init}}^{\textrm{T}}~u_0^{{{k+1}^{\textrm{T}}}}\right]^{\textrm{T}}$.}
 \State{{1.b}~For $t\!=\!1,\ldots,N^{k-1}\!-\!1$, $y_t^{k+1}
 =\argmin_{y_t}~\!\!\!\mathcal L_t.$}
 \State{{1.c}~For $t\!=\!N^{s}$, $x_t^{{k+1}}
 =\argmin_{x_t}~\mathcal L_N$ and $y_t^{k+1}\!:=\!\left[x_t^{{{k+1}^{\textrm{T}}}}~0\right]^{\textrm{T}}$.}
 \State{{2.}~Compute $\alpha^{k+1}=\frac{1+\sqrt{4\alpha^{k^2}+1}}{2}$.}
 \State{{3.}~For $t=0,\ldots,N^s$, $\sigma_t^{k+1}=\mathbf{Pr}_{\mathbb C}\left(G\hat y_t^{k+1}-d-\frac{1}{\tau}\hat \lambda_t^k\right)$.}
 \State{{4.}~ For $t=1,\ldots,N^s$, 
 $z_t^{k+1} = \frac{H_1\hat y_t^{k+1}+H_2 \hat y_{t-1}^{k+1}}{2}-\frac{\hat w_t^k+\hat v_t^k}{2\tau}$.}
 \State{{5}~For $t=0,\ldots,N^s$ compute \begin{align*} 
 \mu_t^{k+1} &= \hat \mu_t^{k}-\tau \tilde H\begin{bmatrix}
 y_{t-1}^{k+1}\\
 y_{t}^{k+1}\\
 z_t^{k+1}\\
 \sigma_t^{k+1}
 \end{bmatrix} + \tau\begin{bmatrix}
 0\\ 0\\ d
 \end{bmatrix},\\
  \hat \mu_t^{k+1} & = \mu_t^{k+1} +\frac{\alpha^k-1}{\alpha^{k+1}}\left(\mu_t^{k+1}-\mu_t^{k}\right).
 \end{align*}}
 \EndFor
 \caption{FAMA for Problem~\eqref{eq:clqr_time_splitting}.}
 \label{alg:fama_splitting_solver}
 \end{algorithmic} 
\end{algorithm}
\captionsetup[algorithm]{font=small}
\begin{algorithm}[t]
\fontsize{8}{8}\selectfont
\begin{algorithmic} 
\State{Given $x_{\textrm{init}}$, $N_{\textrm{init}}$, $H_1$, $H_2$, $Q$, $R$,
$P$, $G$, $g$, $\tau$, $\eta$, and $\bar k$, $\mathcal X_f$, $K$, and $\tilde
H$.}
\State{Set $N^0 = N_{\textrm{init}}$, $\alpha^0=1$, $\alpha^1 =
{(1+\sqrt{5})}/{2}$, $\mu_t^0=\hat \mu_t^1=\mu_t^{\textrm{start}}$, and $s=0$.}
\If {$x_{\textrm{init}}\in \mathcal X_f$}
\vspace{0.02in}
\State{1. $u^{\infty}=Kx_{\textrm{init}}$, $N = 0$.}
\Else 
\While{Termination criteria are not met}
\vspace{0.02in}
\vspace{0.02in}
  \State{2. $\left [\mathbf y^{s+1}, \hat{\boldsymbol \mu}^{s+1}, \boldsymbol \mu^{s+1}, \alpha^{s+1}\right ] =$\verb!fama!$\left(\hat {\boldsymbol\mu}^{s},{\boldsymbol\mu}^{s}, \alpha^{s}, N^s,\bar k\right)$.}
 \vspace{0.02in}
 \If{$x_N^{s+1}\in\mathcal X_f$}
 \vspace{0.02in}
  \State{3. $N^{s+1} = N^{s}-1$.}
  \vspace{0.02in}
  \State{4. Remove Subproblem~$N^{s}$. }
  \vspace{0.02in}
  \Else
  \vspace{0.02in}
  \State{5. $N^{s+1} =N^{s}+1$.}
  \vspace{0.02in}
  \State{6. Add Subproblem~$N^{s+1}$.}
  \vspace{0.02in}
  \EndIf
 \State{7. $s=s+1$.}
  \vspace{0.02in}
 \EndWhile
 \State{8. $u^{\infty} = u_0$, $N = N^s$.}
 \EndIf
 \caption{CLQR for Problem~\eqref{eq:clqr_time_splitting}.}
 \label{alg:fama_splitting}
 \end{algorithmic} 
\end{algorithm}
Algorithm~\ref{alg:fama_splitting} details our strategy to estimate the solution of Problem~\eqref{eq:clqr_time_splitting} and the length of the horizon $N$. In particular, 
Algorithm~\ref{alg:fama_splitting} relies on Algorithm~\ref{alg:fama_splitting_solver}, which is Algorithm~\ref{alg:fama} applied to Problem~\eqref{eq:clqr_time_splitting}, to 
compute the primal and dual variables (step 2). In particular, (step 1 of Algorithm~\ref{alg:fama_splitting_solver}) $\mathcal L_t$ is the Lagrangian associated with
 Problem~\eqref{eq:clqr_time_splitting} and defined as follows:
\begin{align*}
\mathcal L_t = \min&~ f(y_t)+g({\sigma_t,z_t})+\langle\lambda_t,-Gy_t-\sigma_t+d\rangle\\
&+\langle w_t, z_t-H_1y_t\rangle+\langle v_{t+1},z_{t+1}-H_2y_{t}\rangle.
\end{align*} 
Furthermore, (step 5 of Algorithm~\ref{alg:fama_splitting_solver}) $\tilde H$ is the matrix associated with the multiplier update, i.e., 
\begin{equation*}
\tilde H := \begin{bmatrix}
 0   \!\!&\! H_1 \!\!&\! -I_n \!\!&\! 0\\
 H_2 \!\!&\! 0   \!\!&\!-I_n  \!\!&\! 0\\
 0   \!\!&\! -G  \!\!&\! 0    \!\!&\!-I_p
 \end{bmatrix}.
\end{equation*}
For $t=0$ only $\lambda_0$ is updated, given that $z_t, w_t,$ and $v_t$ are defined for $t=1,\ldots,N^{}$.

It is evident that compared to Algorithm~\ref{alg:fama_splitting_solver},
Algorithm~\ref{alg:fama_splitting} has an additional \verb!if! condition used to check every $\bar k \geq 1$ iterations whether 
$x_{N^{s}}$$\in\mathcal X_f$. If $x_0\in\mathcal X_f$ the algorithm terminates
immediately (step 1). Otherwise, the algorithm terminates only when $x_{N^{s
}}$$\in\mathcal X_f$ and $\{\mathbf y$, $\boldsymbol \mu\}$ returned by
Algorithm~\ref{alg:fama_splitting_solver} reaches a desired accuracy. If
$x_{N^{s }}$$\in\mathcal X_f$, the algorithm decreases $N^s$.
From the splitting perspective, this means that the last subproblem is
removed (steps 3 and 4). Note that removing a subproblem with its associated
dual variables does not compromise the future updates of the remaining subproblems. If $x_{N^{s }}$$\notin\mathcal X_f$, $N^s$ increases by 1 with respect to the previous iterate 
(steps 5 and 6). From the splitting perspective, this means that a new
subproblem (of the same dimension as the previous ones) is added.
\begin{rem}
\label{rem:update}
In theory, we can set $\bar k = 1$, i.e., the algorithm checks the state of the last subproblem at every iteration. In practice, we noticed that checking the state of the last 
subproblem at every iteration affects the convergence of $N^s$ to $N^{\infty}$ given that $N^s$ oscillates around $N^{\infty}$ requiring an higher number of iterations. 
If we allow larger $\bar k$ the oscillations disappear and $N^s$ converges faster to $N^{\infty}$. From the FAMA perspective, a larger $\bar k$ means that at each outer iteration
 $s$ of Algorithm~\ref{alg:fama_splitting}, Problem~\eqref{eq:clqr_time_splitting} is solved up to a given accuracy (which depends on $\bar k$). Then, the quality of the estimates
  is refined every $\bar k$ iterations together with $N^s$. 
\end{rem}
\begin{rem}
\label{rem:modified_N}
Step 8 of Algorithm~\ref{alg:fama_splitting} can be modified to achieve a tighter upper bound on $N^{\infty}$. In particular, by using
 $\mathbf u_N = \{u_0^{(0)},\ldots,u_{N^{}-1}^{(N^{}-1)},Kx_{N^{}}\}$ we can compute $x_{N}^{(0)}$. Then, while $C x_{N}^{(0)}+ Du_N^{(0)}< d$, $N\leftarrow N-1$. 
 The first time the constraints are active, the algorithm terminates. This does not affect the computational time of the algorithm (given that the solution has been already computed), 
 but can improve the initial guess on the length of the horizon for the next problem instance in a closed-loop implementation.
\end{rem}
According to~\cite{Pu2014} the following result concerning the quality of the primal estimates holds:
\begin{thm}
\label{thm:pu2014}
Consider Problem~\eqref{eq:clqr_time_splitting}. Let $\{\mathbf y^k\}$ and $\{\boldsymbol \mu^k\}$ be generated by~Algorithm~\ref{alg:fama_splitting}. If 
Assumption~\ref{ass:positive_def} is satisfied, then, for any $s\geq 0$ and $\bar k\geq 0$, the following holds:
\begin{align*}
\mathbf D(\boldsymbol\mu^*)-\mathbf D (\boldsymbol \mu^s) \!\leq\! \frac{2\operatorname{eig}_{\max}(H_{\mathbf y})}{\sigma_f(s\bar k+1)^2} \|
\boldsymbol \mu^0-\boldsymbol\mu^*\|^2,
\end{align*}
where $\boldsymbol\mu^0$ and $\boldsymbol \mu^*$ are the initial and optimal values of multipliers, respectively, and $H_{\mathbf y}$ is defined as follows:
\begin{equation*}
H_{\mathbf y}:= \operatorname{blockdiag}\{\underbrace{H_1,\ldots,H_1}_{ N^{}}, \underbrace{H_2,\ldots,H_2}_{ N^{}}, \underbrace{-G,\ldots,-G}_{ (N^{}+1)}\},
\end{equation*} 
If $\lambda_t^0\in \mathbb C$ ($t=0,\ldots, N^{}$) and $\boldsymbol y^0$ is such that the consensus constraints are satisfied, then the dual iterates will remain feasible for 
all $k\geq 1$ and
\begin{align}
\label{eq:thm_pu_2}
\|\mathbf y^s-\!\mathbf y^*\|^2 \!\leq\! \frac{4\operatorname{eig}_{\max}(H_{\mathbf y})}{\sigma_f(s\bar k+1)^2}\|
\boldsymbol \mu^0-\boldsymbol\mu^*\|^2.
\end{align}
\end{thm}
\begin{proof}
The proof follows from the one of Theorem 5.3 in~\cite{Pu2014} applied to Problem~\eqref{eq:clqr_time_splitting}.
\end{proof}
Then, the following result holds:
\begin{thm}
\label{thm:ferranti}
Consider Problem~\eqref{eq:clqr_time_splitting}. Let $N$ and $\mathbf u_N = \{u_0^{(0)},\ldots,u_{N^{}-1}^{(N^{}-1)},Kx_{N^{}}\}$ be generated 
by~Algorithm~\ref{alg:fama_splitting}. Under the same assumptions of Theorem~\ref{thm:pu2014} then, for all $s\geq 1$ and $\bar k\geq 1$, there exists $ N^{\infty}\leq N$ 
finite such that  
\begin{align}
\label{eq:thm3}
(\mathbf u_{N^{\infty}}^{s}\!\!-\!\bar{\mathbf u}^{\infty})^{\trans}\!\tilde B(\mathbf u_{N^{\infty}}^{s}\!\!-\!\bar{\mathbf u}^{\infty})\!\leq\!\!
 \frac{4\operatorname{eig}_{\max}(H_{\mathbf y})}{\sigma_f(s\bar k+1)^2}\|
\boldsymbol \mu^0\!\!\!-\!\boldsymbol\mu^*\!\|^2,
\end{align}
where $\bar{\mathbf u}^{\infty} = \{u_0,\ldots,u_{N^{\infty}}\}$ is the (truncated) solution of the CLQR Problem and $\tilde B$ is defined as follows:
\begin{equation*}
\tilde B:=\begin{bmatrix}
I_m & 0 & \ldots & 0\\
B  & I_m &  & 0\\
\vdots &\vdots & \ddots & 0\\
A^{N^{\infty}-1}B & A^{N^{\infty}-2}B& \ldots& I_m  
\end{bmatrix}.
\end{equation*}
\end{thm} 
\ifpaper
\begin{proof}
Given that we initialize $\mathbf y^0$ to achieve consensus, according to Theorem~\ref{thm:pu2014}, $\mathbf z^s$ will remain in consensus, and solving 
Problem~\eqref{eq:clqr_time_splitting} becomes equivalent to solving~Problem~\eqref{eq:IH-MPC}. Hence, we can use the results of Theorem~\ref{thm:scokaert} to
 show that $N^{\infty}$ is finite. Concerning the inequality above, first, note that
$x_t^s-x_i^{\infty} = Ax_0+A^{t-1}B u_0^s+\ldots+Bu_{t-1}^s-Ax_0-A^{t-1}B u_0^{\infty}-\ldots-Bu_{i-1}^{\infty} = A^{t-1}B (u_0^s-u_0^{\infty}) +
\ldots+B(u_{t-1}^s-u_{t-1}^{\infty})$. Second, according to Theorem~\ref{thm:scokaert}, $\mathbf u^{*}=\mathbf u^{\infty}$. Hence, the inequality~\eqref{eq:thm3} follows directly 
from~\eqref{eq:thm_pu_2}. Third, for $s\bar k\to \infty$, we can conclude that the the control sequence obtained using Algorithm~\ref{alg:fama_splitting} converges to the CLQR control
 law obtained by solving Problem~\eqref{eq:clqr}.
\end{proof}
\else 
\begin{proof}
Refer to~\cite{FerrantiReport}
\end{proof}
\fi
\begin{rem}
Note that Theorem~\ref{thm:ferranti} considers the truncated sequences $\mathbf u_N$ and $\bar{\mathbf u}^{\infty}$ for practical reasons, given that, after 
$N^{\infty}$ steps the control commands are obtained using the LQR gain $K$, i.e., are identical for both sequences.
\end{rem}
Initializing the new values of the multipliers is important to satisfy the assumptions of the theorems above. In the estimation phase of the horizon length, when $N^s$ 
increases then setting 
$z_{N^s}^s:=H_2y_{N^s-1}^s$ (i.e., to maintain consensus between the former last subproblem and the new last subproblem) allows one to initialize
 $w_{N^{k}}^{k-1} =v_{N^{s}}^{s-1}= 0$. 
 Concerning $\lambda_{N^s}^{s-1}$, any value such that $\lambda_{N^s}^{s-1}\in
 \mathbb C$\ifpaper~(according to Lemmas~\ref{lem:Pu1}
 and~\ref{lem:Pu2})~\else~(according to Lemmas~3.4 and 3.5 in~\cite{Pu2014})~\fi
 can be used\ifpaper~(e.g., {\small $\lambda_{N^s}^{s-1} =
 \lambda_{N^s-1}^{s-1}$}).\else.\fi

In~\cite{StathopoulosTAC2015} splitting strategies are also used to estimate
$N$.
Compared to~\cite{StathopoulosTAC2015}, we use a different strategy to compute the length of the prediction horizon online. First, we do not propagate the dynamics forward at each 
iterate (until a stable value of $N\geq N^{\infty}$ is reached), but we
  check whether the state of the last subproblem is in $\mathcal X_f$ (an
  inexpensive operation). Second, thanks to the time splitting, our algorithm
  allows one to decrease the length of the horizon online, while the approach
  proposed in~\cite{StathopoulosTAC2015} allows the horizon value only to increase from its initial guess. Specifically, if our initial guess is too conservative, 
  Algorithm~\ref{alg:fama_splitting} starts removing the tail subproblems. Removing subproblems implies removing dual variables that could, in general, affect the future updates of
   the algorithm, such as in~\cite{StathopoulosTAC2015}. This is not the case for the time splitting. The dual variables of the subproblems are independent of each other and removing 
   one of them (on the tail of the horizon) does not compromise the future updates of the others.

Our approach relies on the results of~\cite{Scokaert1998}. In~\cite{Scokaert1998}, however, the steps of the algorithm are more involved from the computational point of view. 
Step 1 of Algorithm~\ref{alg:scokaert} requires the solution of a constrained QP of dimension proportional to the length of the prediction horizon. Recursively solving a constrained
 QP can be extremely time-consuming, especially when the computational resources are limited, such as, in embedded applications. Our design, thanks to the time splitting, only 
 solves unconstrained least-squares problems and simple gradient calculations. Furthermore, an increase in the length of the horizon does not change the dimension of the subproblems, 
 given that their size is independent of the length of the prediction horizon. 

\ifpaper The proposed algorithm can be, in principle, fully parallelized. If $N$
independent workers are available, each of them can be dedicated to a subproblem. The workers communicate with their neighbors only at given time instances to exchange information 
concerning the consensus variables. When the number of subproblems is large and the number of workers is smaller than $N^{}$ asynchronous update strategies can be beneficial to improve
 the performance of the algorithm. Investigation of asynchronous update strategies is part of our future work.
\else 
\fi
\begin{rem}
According to Theorem~\ref{thm:ferranti}, Algorithm~\ref{alg:fama_splitting} returns an optimal solution for Problem~\eqref{eq:clqr} for $s\bar k\to \infty$.
In practical implementations, the algorithm terminates after a fixed number of iterations. In this scenario, constraint tightening techniques can be used to enforce the
 feasibility of the primal estimates (e.g.,~\cite{Rubagotti2014}). 
\end{rem} 
\section{NUMERICAL EXAMPLE}
\label{sec:sec_simulation_results.tex}
We tested our design on the system proposed in~\cite{StathopoulosTAC2015}.
\ifpaper
The system is described by the following matrices:
\begin{equation}
 A = \left[\begin{array}{cc} 1.1 & 2 \\ 0& 0.95\end{array}\right],~ B= \left[\begin{array}{c}
  0\\ 0.0787
 \end{array}\right].
\end{equation}
The state and the control input are constrained in $\mathcal X := \{x\in \mathbb R^n\,|\,\|x\|_{\infty}\leq 10\}$ and $\mathcal U:= \{u\in \mathbb R^m\,|\, \|u\|_{\infty}\leq 1\}$, 
respectively. The matrices $Q$ and $R$ are the identity matrices to satisfy Assumption~\ref{ass:positive_def}. 
\else 
\fi
We computed offline the maximal positively invariant set $\mathcal
X_{f}^{\max}\!\!$ of the closed
loop associated with the LQR controller with weighting matrices $Q$
and $R$ according to~\cite{MPT3}. Then, we selected $X_f$ to be a tightened
subset of $\mathcal X_{f}^{\max}\!$ to take into account the early termination
of the solver. In particular, we tightened the terminal set by a quantity $\epsilon = 10^{-3}$ proportional to the selected termination criterion for the algorithm $\|\boldsymbol {\mu}^k-\boldsymbol {\mu}^{k-1}\|^2\leq \epsilon$. Then, we tuned $\tau  = 0.0726$ and set $\bar k = 1000$ for $s=0$ and $\bar k =1$ for $s>0$.

We tested Algorithm~\ref{alg:fama_splitting} for 1592 different initial
conditions uniformly sampled from $\mathcal X$. For each initial condition, we ran the proposed algorithm and, at the end of each simulation, we tested whether the state $x_{N} = A^{N} x_0 + A^{N-1}B u_0^{(0)} + \ldots+B u_{N-1}^{(N-1)}$ was in $\mathcal X_f$. For practical reasons, we terminated the simulation if convergence to the suboptimal solution was not achieved within $k_{\max}=10^5$ iterations.
  
First, we compared the behavior of the algorithm without and with the backtracking of the horizon length described in Remark~\ref{rem:modified_N}. Figure~\ref{fig:hist_comparison} 
compares the horizon length obtained in the two scenarios using $N^0=20$ as
initial guess on the horizon length. The proposed backtracking strategy reduces the conservatism in the value of $N$. As part of our future work, we plan to investigate an online
 strategy to reduce the horizon length in steps 3-6 of Algorithm~\ref{alg:fama_splitting} (for example, by replacing the \verb!if! condition with a \verb!while!).

Second, we compared the behavior of Algorithm~\ref{alg:fama_splitting} with backtracking for three different initial guesses on the horizon length, i.e., $N^0 \in\{2, 8, 20\}$. 
Figure~\ref{fig:hist} shows the distributions of the estimated values of $N$ for the three aforementioned scenarios. Furthermore, the vertical dashed lines in Figure~\ref{fig:hist} 
represent the mean value of the horizon length computed using the proposed algorithm. Note that we removed from the plots the initial conditions that were inside the terminal set and 
the initial conditions that lead the algorithm to terminate after $k_{\max}$ iterations. Hence, based on this selection on the initial states, we noticed that for $N^0=2$ 
Algorithm~\ref{alg:fama_splitting} terminates within $k_{\max}$ in $238$ cases, for $N^0=8$ it terminates in $818$ cases, and for $N^0 = 20$ it terminates in $1592$ cases.
 Note that by using $N^0 =2$ we are only able to converge within $k_{\max}$
 iterations only when the optimal horizon length is close to 2. This observation
 confirms the benefits of warm-starting the horizon length. 
 Figure~\ref{fig:iter} shows the average number of iterations $s\bar k$ needed 
 to compute $N\geq N^{\infty}$ and an optimal solution for
 Problem~\eqref{eq:clqr}. Note that, warm starting the length of the horizon
does not compromise the number of iteration needed by the algorithm to converge
to $N\leq N^0$.
\ifpaper We believe that the current values can be improved if the back tracking
is implemented along with the computation of the primal and dual variables.
\else
\fi
\begin{figure}[t] 
  \includegraphics[width=.95\linewidth]{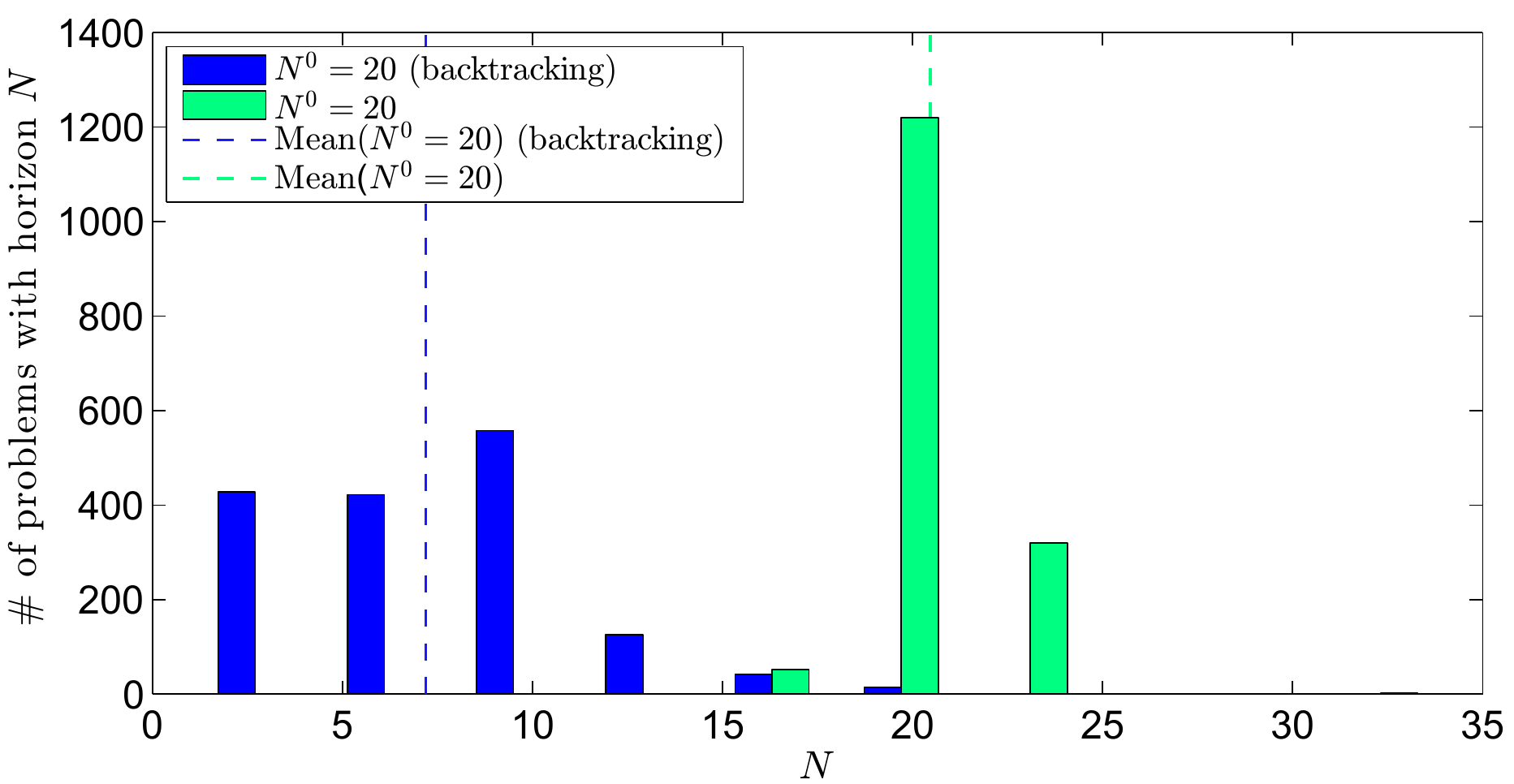}
  \caption{Comparison of the values of $N\geq N^{\infty}$ obtained using Algorithm~\ref{alg:fama_splitting} without and with backtracking, respectively.}\label{fig:hist_comparison}
  \end{figure}
 \begin{figure}[t] 
   \includegraphics[width=.95\linewidth]{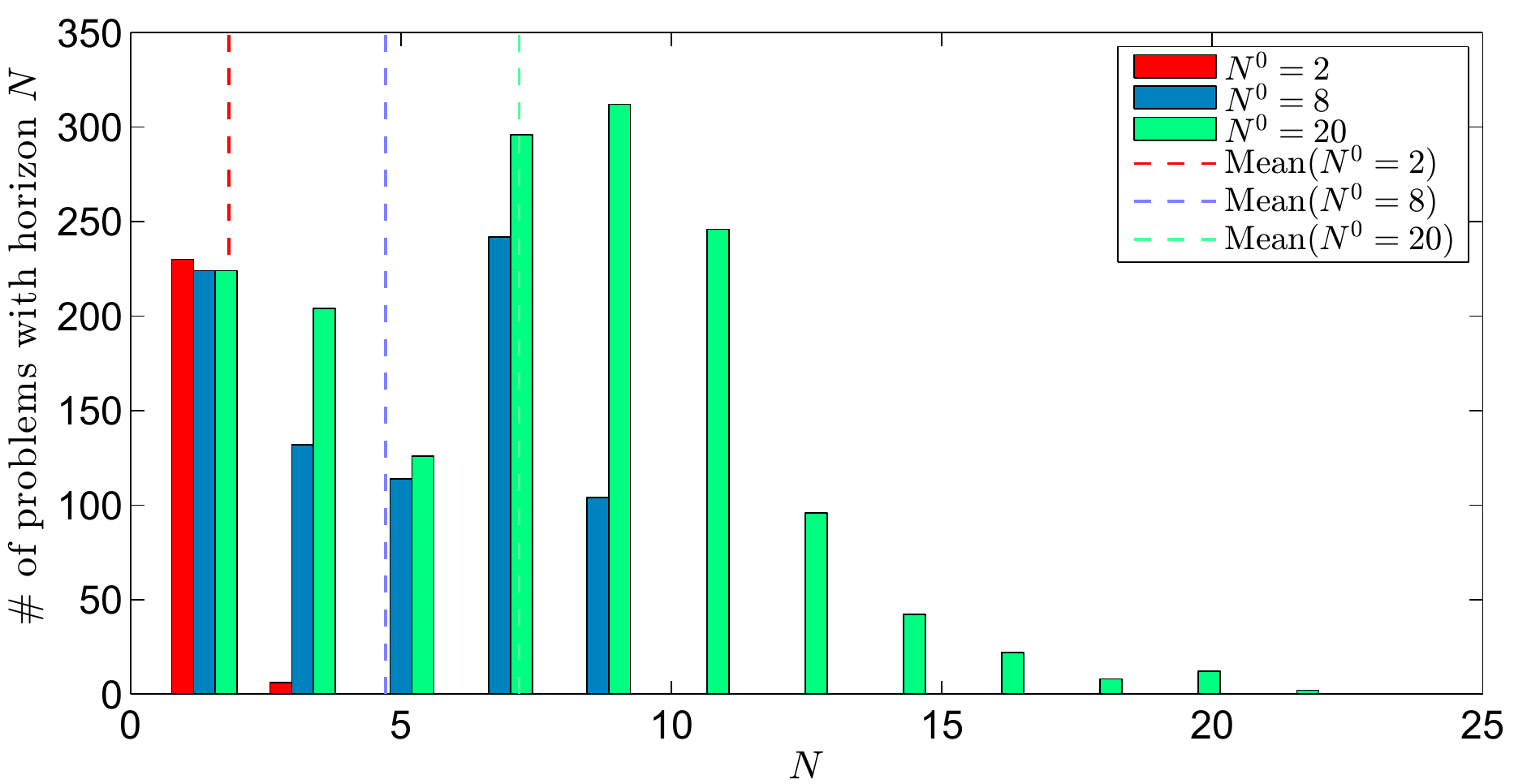}
   \caption{Estimated horizon length $N$ computed using
  Algorithm~\ref{alg:fama_splitting} with backtracking starting from different initial conditions in $\mathcal X$ using different values of $N^0$ to initialize the algorithm.}\label{fig:hist}
\end{figure}
  \begin{figure}[t]
  \includegraphics[width=.95\linewidth]{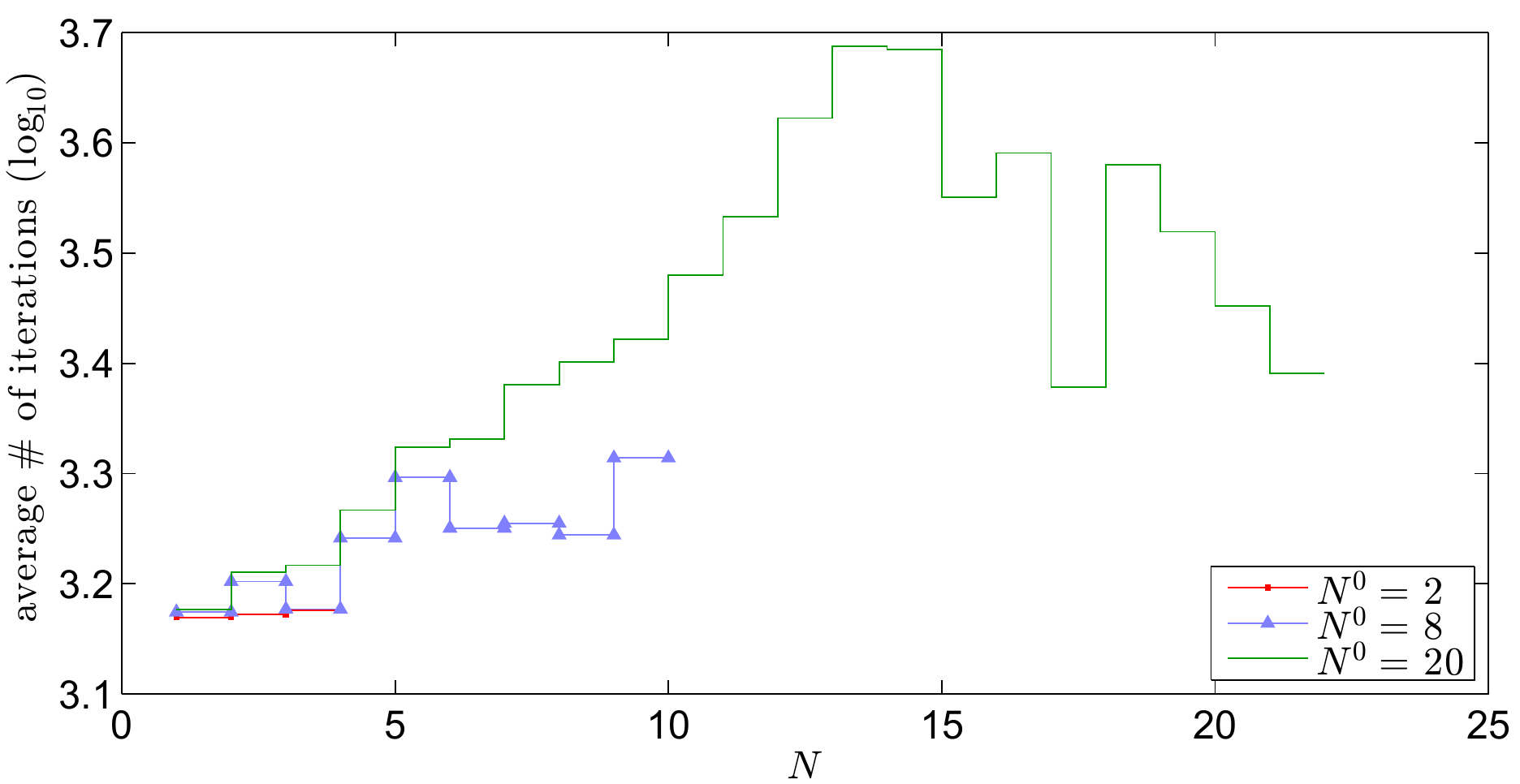}
  \caption{Average number of iterations of
  Algorithm~\ref{alg:fama_splitting} obtained for each initial condition with respect to the length of the horizon $N$ using different values of $N^0$ to initialize the algorithm.} \label{fig:iter}
\end{figure} 
\section{CONCLUSIONS}
\label{sec:sec_conclusions}
This paper proposes an alternative approach to solve the constrained linear quadratic regulator (CLQR) problem using operator splitting techniques.
 The original CLQR problem is reformulated as an MPC problem, whose horizon length is estimated online for each initial condition. We show that the solution
  obtained using our proposed algorithm is an optimal solution of the CLQR problem and that the horizon length is bounded. Finally, we tested our design on a planar system
   to show the advantages of the proposed technique that allows to reduce the number of iterations needed to achieve an optimal solution for the CLQR problem thanks to the warm 
   starting of the horizon length.    

As part of our future work, we plan investigate the possibility of asynchronous updates. Furthermore, we plan to test the proposed algorithm on a practical application.
 

\end{document}